\documentclass[EJP]{ejpecp} % replace ECP by EJP if needed.  

\usepackage{enumerate}  % uncomment to use this package

\SHORTTITLE{A compact containment result for nonlinear historical superprocess approximations} 

\TITLE{A compact containment result for nonlinear historical superprocess approximations for population~models~with~trait-dependence} % \thanks is optional. Insert line breaks with \\

\AUTHORS{%
  Sandra~Kliem\footnote{Fakult\"at f\"ur Mathematik, Universit\"at
    Duisburg-Essen, %Thea-Leymann-Str. 9, D-45127 Essen,
    Germany. \EMAIL{sandra.kliem@uni-due.de}}}%AUTHORS
\KEYWORDS{nonlinear historical superprocess; interacting particle systems; compact containment; tightness; exponential rates; evolution model; measure-valued processes on c\`adl\`ag functions; Polish space} % Separate items with ;

\AMSSUBJ{60J80} % Edit. Separate items with ;
\AMSSUBJSECONDARY{60G57; 60J68; 60K35} % Optional, separate items with ;

\SUBMITTED{May 8, 2014} % Edit.
\ACCEPTED{September 21, 2014} % Edit.

\ARXIVID{1405.0815v1} % Edit.
\VOLUME{19}
\YEAR{2014}
\PAPERNUM{97}
\DOI{v19-3506}

\ABSTRACT{We consider an approximating sequence of interacting population models with branching, mutation and competition. Each individual is characterized by its trait and the traits of its ancestors. Birth- and death-events happen at exponential times. Traits are hereditarily transmitted unless mutation occurs. The present model is an extension of the model used in \cite{MT2012-2}, where for large populations with small individual biomasses and under additional assumptions, the diffusive limit is shown to converge to a nonlinear historical superprocess. The main goal of the present article is to verify a compact containment condition in the more general setup of Polish trait-spaces and general mutation kernels that allow for a dependence on the parent's trait. As a by-product, a result on the paths of individuals is obtained. An application to evolving genealogies on marked metric measure spaces is mentioned where genealogical distance, counted in terms of the number of births without mutation, can be regarded as a trait. Because of the use of exponential times in the modeling of birth- and death-events the analysis of the modulus of continuity of the trait-history of a particle plays a major role in obtaining appropriate bounds.}

\newcommand{\DD}{\mathbb{D}}
\newcommand{\EE}{\mathbb{E}}
\newcommand{\NN}{\mathbb{N}}
\newcommand{\PP}{\mathbb{P}}
\newcommand{\RR}{\mathbb{R}}

\newcommand{\CB}{\mathcal{B}}
\newcommand{\CC}{\mathcal{C}}
\newcommand{\CD}{\mathcal{D}}

\newcommand{\CF}{\mathcal{F}}
\newcommand{\CI}{\mathcal{I}}

\newcommand{\CK}{\mathcal{K}}

\newcommand{\CM}{\mathcal{M}}

\newcommand{\CZ}{\mathcal{Z}}

\newcommand{\1}{\mathbb{1}}

\newcommand{\eqn}[1]{\begin{equation} #1 \end{equation}}
\newcommand{\eqan}[1]{\begin{align} #1 \end{align}}
\newcommand{\lbeq}[1]{\label{#1}}
\newcommand{\nn}{\nonumber}

\begin{document}

%%                                                               %%
%% No need for \maketitle.                                       %%
%%                                                               %%

%%                                                               %%
%% Please replace what follows by the body of your article       %%
%% (up to the bibliography):                                     %%
%%                                                               %%

% ======================================================================

\section{Introduction}

% ======================================================================

The main goal of the present article is to verify a compact containment condition for \textit{ ``Nonlinear historical superprocess approximations for population models with past dependence''} as stated in \cite[Lemma 3.5(i)]{MT2012-2} in a more general setup. As a by-product, two errors from \cite{MT2012-2} are fixed and a result on the paths of individuals is obtained in a broader context. The obtained result extends the results of \cite{MT2012-2} as far as compact containment of the approximating processes goes.

Compact containment is one of the two properties to establish tightness of a sequence of laws on $\DD_Y([0,T])$, where we denote by $\DD_Y([0,T])$ the space of c\`adl\`ag functions from $[0,T]$ to $Y$ embedded with the Skorohod topology, with $Y$ a given Polish space (cf. Jakubowski's criterion for tightness as stated in \cite[Theorem 3.6.4]{bD1993}). Compact containment means that for any $T>0$ and for any $\epsilon>0$ fixed, one can find a compact set in $Y$ such that the $n^{\mbox{th}}$-approximating population at time $t \in [0,T]$ is located outside this set with probability $\epsilon$ at most, uniformly in time $t \in [0,T]$ and $n \in \NN$. The result is stated in Theorem~\ref{THM:MT-extended} (here $Y = \CM_F(\DD_E)$). Consequently, compact containment results provide additionally some control on the paths of particles (cf. %Corollary~\ref{COR:explicit-bound} and 
Lemma~\ref{LEM:app_to_mod_of_cont}).

In \cite{MT2012-2}, interacting population models are under consideration, where each individual is assigned a trait. The models involve branching, mutation and competition. Birth- and death-events happen at exponential rates. The rates depend on the trait of the individual and on the history of its trait through its ancestry. We therefore identify each individual rather by the history of its past traits up to the present time, that is we consider \textit{historical particles}. Competition is modeled by means of an additional term in the death-rate that takes into account the trait-history of the other individuals as well.

As a consequence, historical processes are particularly well-suited to record the evolution of the traits of individuals in a population over time. For each $n \in \NN$, an approximating population is given. It is then shown in \cite{MT2012-2} for large populations with small individual biomasses and under additional assumptions, that the diffusive limit for $n \rightarrow \infty$ converges to a limiting nonlinear historical superprocess limit. Existence of the limiting process is established by proving the tightness of the sequence of laws of the approximating populations.

One of the major strengths of historical processes is that they allow for a control of the traits of \textbf{historical} particles, present in the population at time $t$, uniformly in $t \in [0,T]$. That is, we obtain a control on the history of the trait of the particle through its ancestry as well. For instance, Lemma~\ref{LEM:app_to_mod_of_cont} yields a control on the modulus of continuity of the paths of the particles in the population.

The present article extends the result on compact containment from \cite{MT2012-2} from $\RR^d$ to Polish trait-spaces and from translation-invariant Gaussian mutation densities to a class of mutation kernels that allow for a dependence on the parent's trait as well (cf. Hypothesis~\ref{HYP:mut-kernel}). Additionally, a lower bound on the interactive killing rate is dropped (cf. \eqref{ass-on-U}). 

Historical particles can be modeled as c\`adl\`ag paths on the trait space $E$. Here it is important to recall that relative compactness in $\DD_E$ involves both controlling the range as well as the modulus of continuity of the traits of the particle along its path. To show compact containment of the sequence of approximating particle systems therefore involves not only controlling jump-sizes in trait space at times of birth but also controlling the impact of an accumulation of jumps in a period of time. 

As a result, the use of exponential rates in modeling birth- and death-events is a challenge compared to a setup with equidistant time-steps. Indeed, one of the main steps in proving compact containment is to obtain a bound on the expected fraction of historical particles at a fixed time $T$ outside a compact set $K \subset \DD_E$ (cf. Proposition~\ref{PRO:MT-extended} to come). In the equidistant case, this bound can be readily obtained by induction over each time-step respectively birth/death-event, see \cite[Lemma II.3.3(a)]{bP2002} and reduces to a bound on the evolution in trait-space of a single particle only. In the non-equidistant setup, the number of trait-changes (that is, birth with mutation) until time $T$ now plays a major role in the derivation of an appropriate bound. 

Coupling-techniques are an important tool in this article: For the populations in question, we construct couplings with ``dominating'' respectively ``minorizing'' populations in the sense that one population is a sub/super-population of the other one. This is done by choosing birth- and death-rates appropriately with the aim to loose certain dependencies in the rates on the paths of the particles. Also, for two paths with a different number of trait-changes until time $T$, the moduli of continuity are compared by means of coupling-techniques.
 
In \cite{KW2014} the results of the present paper will be applied in a context of evolving genealogies where the metric is the mutational distance, with mutations happening at birth (see also the remark at the end of Section~\ref{S:results}).

We next briefly introduce the model of \cite{MT2012-2} and its extensions in the present context. For a biological motivation and a discussion of literature the interested reader is referred back to \cite{MT2012-2}. We start with some basic notation taken in part from \cite{MT2012-2}.

% ----------------------------------------------------------------------

\begin{notation}[cf. Notation of {\cite[Section 1]{MT2012-2}}]
For a given metric space $E$, we denote by $\CC(E), \CC_b(E)$ respectively $\CB(E)$ the continuous, bounded continuous respectively bounded functions on $E$. 

We further denote by $\DD_E = \DD(\RR_+,E)$ the space of c\`adl\`ag functions from $\RR_+$ to $E$ embedded with the Skorohod topology. For a function $x \in \DD_E$ and $t>0$, we denote by $x^t$ the stopped function defined by $x^t(s) = x(s \wedge t)$ and by $x^{t-}$ the function defined by $x^{t-}(s) = \lim_{r \uparrow t} x^r(s)$. We will also often write $x_t=x(t)$ for the value of the function at time $t$. For $y, w \in \DD_E$ and $t \in \RR_+$, we denote by $(y|t|w) \in \DD_E$ the following path:
\eqn{
\lbeq{concatenated-path}
  (y|t|w) = \begin{cases}
              y_u & \mbox{ if } u<t \cr
              w_{u-t} & \mbox{ if } u \geq t. \cr
            \end{cases} 
}
For a constant path $w$ with $w_u=x, \ \forall u \in \RR_+$, we will write $(y|t|x)$ with a slight abuse of notation.

Denote by $\CM_F(E)$ the set of finite measures on $E$ embedded with the topology of weak convergence.
\end{notation}

% ======================================================================

\section{The historical particle system}

% ======================================================================

We shortly introduce the population model from \cite{MT2012-2}. Where there is an extension made, we will remark on it. Note in particular that \cite{MT2012-2} prove existence and convergence to a nonlinear historical superprocess. As we only concern ourselves with proving a compact containment condition for the approximating populations, certain assumptions made in \cite{MT2012-2} become therefore unnecessary.

% ----------------------------------------------------------------------

In the $n^{th}, n \in \NN$ approximation step, \cite{MT2012-2} consider a discrete population in continuous time where individuals reproduce asexually and die. Each individual is assigned a trait. The first extension is that
\eqn{ 
  \mbox{the trait space } E \mbox{ is assumed to be Polish } 
}
contrary to \cite[paragraph before (2.1)]{MT2012-2}, where $E$ is restricted to $E = \RR^d$.

The lineage or past history of an individual is defined as follows: To an individual of trait $x$ born at time $S_m$, having $m-1$ ancestors born at times $0=S_1 < S_2 < \cdots < S_{m-1},$ with $S_{m-1} < S_m$, and of traits $(x_1,x_2,\ldots,x_{m-1})$, we associate the path
\eqn{
\lbeq{lineage}
  y_t = \sum_{j=1}^{m-1} x_j \1_{S_j \leq t < S_{j+1}} + x \1_{S_m \leq t}. 
}
This path is called the \textit{lineage} of the individual. For $n \in \NN$, we consider an individual characterized by the lineage $y \in \DD_{\RR^d}$ in a population $X^n \in \DD_{\CM_F(\DD_E)}$:

The population at time $t$ is represented by a finite point measure
\eqn{
\lbeq{def-X}
  X_t^n = \frac{1}{n} \sum_{i=1}^{N_t^n} \delta_{y^i_. \wedge t} \in \CM_F(\DD_E), 
}
where $N_t^n = n \langle X_t^n , 1 \rangle$ is the number of individuals alive at time $t$. Note in particular that individuals are attributed the weight $1/n$ in this scaling.

\textbf{Initial conditions:} To ensure existence, uniqueness and compact containment of the approximating particle systems, assume
\eqn{ 
\lbeq{tightness-initial}
  \sup_{n \in \NN} \EE[ \langle X_0^n,1 \rangle^2 ] < \infty \quad \mbox{ and } \quad \mbox{ the sequence of laws of } (X_0^n)_{n \in \NN} \mbox{ is tight on } \CM_F(\DD_E). 
}
\textit{The initial conditions coincide with what is used in \cite{MT2012-2} in the parts of proofs that are relevant to our article (cf. \cite[Proposition 2.6 and Proposition 3.4]{MT2012-2}). The corresponding first part of the assumption can be found in \cite[(2.14)]{MT2012-2}. An exponent of $3$ instead of $2$ only becomes necessary in the context of applying a Girsanov-argument along the lines of the proof of \cite[Theorem 5.6]{FM2004}. Note in particular that this bound yields a uniform bound on the first and second moments of the overall mass over time, see Lemma~\ref{LEM:moment-bounds} below, which is not only a crucial ingredient in the proof of existence and uniqueness of the approximating systems but also important in verifying the compact containment condition as we deal with finite measures and not probability measures (see \eqref{def-X}). The second part of the above assumption is included in \cite[(3.5)]{MT2012-2}.}

Let us now recall the population dynamics. 

\textbf{Reproduction:} The birth rate at time $t$ is
\eqn{
  b^n(t,y) = n r(t,y) + b(t,y) 
}
with $r, b \in \CB(\RR_+ \times \DD_E)$ such that
\eqn{
\lbeq{ass-r}
  0 < \underline{R} \leq r(t,y) \leq \overline{R} \quad \mbox{ and } \quad 0 \leq b(t,y) \leq \overline{B}.
}
\textit{In \cite[(2.3)--(2.4)]{MT2012-2} it is additionally assumed that $b, r$ are continuous and that $r$ can be written in the explicit form \cite[(2.4)]{MT2012-2}. These assumptions are not used in the proofs of existence and uniqueness of the approximating processes $X^n$ and in the proof of compact containment and we therefore drop them in our statements.}

When an individual with trait $y_{t-}$ gives birth at time $t$, the new offspring is either a mutant or a clone:
\begin{itemize}
\item
With probability $1-p \in [0,1]$, the new individual is a clone of its parent, with same trait $y_{t-}$ and same lineage $y$.
\item
With probability $p \in [0,1]$, the offspring is a mutant of trait $h$, where $h$ is drawn according to the distribution $\alpha_n(y_{t-},h)$, where
\eqn{ 
  \alpha_n(x,dh), \quad x \in E, h \in E \backslash \{x\} 
}
is a stochastic kernel, the so-called \textit{mutation kernel} on $E$. To this mutant is associated the lineage $(y|t|h)$. 
\end{itemize}
\textit{In \cite[paragraph before (2.5)]{MT2012-2} it is assumed that $\alpha_n(x,x+dh)=k^n(h) dh$, that is the mutant has trait $y_{t-}+h$, where $h$ is drawn according to the distribution $k^n(h) dh$. For the sake of simplicity, the mutation density $k^n(h)$ is assumed to be a Gaussian density with mean $0$ and covariance $\sigma^2\textbf{Id}/n$.} 

Here we generalize from mutation densities to mutation kernels and allow for a dependence on the parent's trait as well. Note that \cite{MT2012-2} often speak of ``jump sizes'' to signify the change of trait at time $t$ from $y_{t-}$ to $y_t=y_{t-}+h$. In the present context we continue to use this wording to signify the change of trait at time $t$ from $y_{t-}$ to $y_t=h$.

% ----------------------------------------------------------------------

\begin{hypothesis}[Assumption on the mutation kernel]
\label{HYP:mut-kernel}
Let $\alpha_n(x,dh), n \in \NN$ be a stochastic kernel on $E$. For $y_0 \in E$ fixed, let $Y^n \in \DD_E$ be a process that starts in $y_0$ and jumps according to the kernel $\alpha_n(x,dh)$ at rate $n$. Denote by $\PP_{y_0}^n$ its distribution starting from $y_0$. We now assume that 
\eqn{
\lbeq{hyp-mut-kernel}
  \mbox{the sequence of laws of } \Big( \int_{\DD_E} X_0^n(dy) \PP^n_{y_0} \Big)_{n \in \NN} \mbox{ is tight on } \DD_E,
}
where $X_0^n$ is the initial condition of the $n^{th}$-approximating population. 
%
%Let $\alpha_n(x,dh), n \in \NN$ be a stochastic kernel on $E$. For $y_0 \in E$ fixed, let $Y^n %\in \DD_E$ be a process that starts in $y_0$ and jumps according to the kernel $\alpha_n(x,dh)$ %at rate $n$. Denote by $\PP_{y_0}^n$ its distribution starting from $y_0$. We now assume that for %all $T, \epsilon>0$ and $\Gamma_0 \subset E$ compact, we can find $K \subset \DD_E$ compact, such %that
%
%\eqn{
%\lbeq{hyp-mut-kernel}
%  \sup_{n \in \NN} \EE \Big[ \int_{y \in \DD_E: y_0 \in \Gamma_0} X_0^n(dy) \PP_{y_0}^n( \{ y: %y^T \not\in K^T \} ) \Big] < \epsilon, 
%}
%
%where $K^T = \{ y^T | y \in K \}$ and $X_0^n$ is the initial condition of the %$n^{th}$-approximating population. 
%
\end{hypothesis}

% ----------------------------------------------------------------------

In Lemma~\ref{LEM:suff-cond-mut-kernel} below we give sufficient conditions on the kernel to satisfy this hypothesis. One of the conditions includes the Gaussian setup from \cite{MT2012-2}.

\textbf{Death:} The death rate at time $t$ is
\eqn{
\lbeq{death-rate} 
  d^n(t,y,X^n) = n r(t,y) + D(t,y) + \int_0^t \int_{\DD_E} U(t,y,y') X^n_{t-s}(dy') \nu_d(ds) 
}
with $D \in \CB(\RR_+ \times \DD_E)$, an interaction kernel $U \in \CB(\RR_+ \times \DD^2_E)$ and a Radon measure $\nu_d$ that satisfy
\eqan{
\lbeq{ass-on-U}
  & \exists \overline{D} > 0, \forall y \in \DD_E, \forall t \in \RR_+, 0 \leq D(t,y) < \overline{D}, \nn\\
  & \exists \overline{U} > 0, \forall y, y' \in \DD_E, \forall t \in \RR_+, 0 \leq U(t,y,y') < \overline{U}. 
}
\textit{Once again we drop the continuity assumptions from \cite[(2.6)--(2.7)]{MT2012-2}. It is important to note that we weaken the second part of assumption \cite[(2.7)]{MT2012-2} on the interaction kernel $U$: \cite{MT2012-2} additionally assume $\exists \underline{U}>0:\ \forall y, y' \in \DD_E, \forall t \in \RR_+, \underline{U} < U(t,y,y')$.}

The proof of existence and uniqueness of the approximating particle systems is a direct adaptation of \cite[Sections 2,3 and 5]{FM2004}. 

% ======================================================================

\section{Results}\label{S:results}

% ======================================================================

In this section we provide results and short proofs, as well as an outlook at the end. We start with a uniform bound on the first and second moments of the overall mass over time, resulting from the assumptions made above.

% ----------------------------------------------------------------------
%
\begin{lemma}
\label{LEM:moment-bounds}
For all $T>0$,
\eqn{
\lbeq{unif-moment-bound}
  \sup_{n \in \NN} \EE\big[ \sup_{t \in [0,T]} \langle X_t^n , 1 \rangle^2 \big] < \infty.
}
\end{lemma}

% ----------------------------------------------------------------------

\begin{proof} The proof is a direct adaptation of the proof of \cite[Theorem 5.6]{FM2004}. \end{proof}

% ----------------------------------------------------------------------

We continue by providing two sufficient conditions on the mutation kernel to satisfy Hypothesis~\ref{HYP:mut-kernel}. The conditions are inspired by \cite[Assumption 2.3]{MT2012-1}.
%
% ----------------------------------------------------------------------

\begin{lemma}
\label{LEM:suff-cond-mut-kernel}
Suppose assumptions \eqref{tightness-initial} on the initial conditions $(X_0^n)_{n \in \NN}$ hold. Either condition on the mutation kernel $\alpha_n(x,dh)$ to follow is then sufficient to satisfy Hypothesis~\ref{HYP:mut-kernel} on $\alpha_n(x,dh)$. 

Let
\eqn{ 
  A^n f(x) := n \int_{E \backslash \{x\}} (f(h)-f(x)) \alpha_n(x,dh). 
}
\begin{enumerate}
\item[(1)]
$E$ is compact and there exists a generator $A$ of a Feller semi-group on $\CC_b(E)$ with domain $\CD(A)$ dense in $\CC_b(E)$ such that
\eqn{
\lbeq{lemma-mut-kernel-1}
  \forall f \in \CD(A), \quad \lim_{n \rightarrow \infty} \sup_{x \in E} \Big| A^n f(x) - A f(x) \Big| = 0.
}
\item[(2)]
$E$ is a closed subset of $\RR^d$ and there exists a generator $A$ of a Feller semi-group on $\CC_b(E)$ with domain $\CD(A)$ dense in $\CC_b(E)$ such that there exists $l_1 \geq l_0 \geq 2$ with $\CC_b^{l_1}(E) \subset \CD(A)$ and such that $\forall f \in \CC_b^{l_1}(E), \forall x \in E$,
\eqn{ 
  | Af(x) | \leq C \sum_{|k| \leq l_0 \atop k=(k_1,\ldots,k_d)} | D^kf(x) | 
}
and
\eqn{
\lbeq{lemma-mut-kernel-2}
  \sup_{x \in E} \Big| A^n f(x) - Af(x) \Big| 
  \leq \epsilon_n \sum_{|k| \leq l_1 \atop k=(k_1,\ldots,k_d)} \| D^kf \|_\infty, 
}
where $D^kf(x) = \partial_{x_1}^{k_1} \cdots \partial_{x_d}^{k_d} f(x)$, $\epsilon_n$ is a sequence tending to $0$ as $n$ tends to infinity and $C$ is a constant.
\end{enumerate}
\end{lemma}

% ----------------------------------------------------------------------

\begin{proof}
Suppose assumptions \eqref{tightness-initial} on the initial conditions $(X_0^n)_{n \in \NN}$ hold. To establish \eqref{hyp-mut-kernel} recall Jakubowski's criterion for tightness (see \cite[Theorem 3.6.4]{bD1993}) to see that it suffices to show 
%Suppose assumptions \eqref{tightness-initial} on the initial conditions $(X_0^n)_{n \in \NN}$ hold. To establish \eqref{hyp-mut-kernel} it therefore suffices to show that the sequence of laws of $Y^n$ is tight on $\DD_E([0,T])$ with $Y^n_0 \sim y_0$, where $y \sim \1_{\{ y_0 \in \Gamma_0 \}} X_0^n(dy)$. Recall Jakubowski's criterion for tightness (see \cite[Theorem 3.6.4]{bD1993}) to see that it suffices to show 

(i) that the sequence of laws of $f \circ Y^n$ is tight on $\DD_\RR$ for all $f \in H$, where $H \subset \CC(E)$ separates points in $E$ and is closed under addition and

(ii) a compact containment condition holds, that is for all $T>0$ and $\eta>0$ there exists $\Gamma \subset E$ compact such that
\eqn{ 
  \inf_n \PP(Y^n(t) \in \Gamma \mbox{ for all } 0 \leq t \leq T) \geq 1-\eta. 
}

In case (1) respectively (2), (i) follows from \cite[Theorem III.9.4]{bEK2005} and \eqref{lemma-mut-kernel-1} respectively \eqref{lemma-mut-kernel-2}. In case (1), (ii) follows by compactness of $E$. In case (2), (ii) follows by adapting the reasoning from \cite[Proof of Lemma 3.3]{MT2012-1}. 
\end{proof}

% ----------------------------------------------------------------------

\begin{remark}
The second case covers the setup of \cite{MT2012-2} where the mutation kernel is assumed to be translation invariant and of the form $k^n(h) dh$ with mutation density $k^n(h)$ a Gaussian density with mean $0$ and covariance matrix $\sigma^2$\textbf{Id}$/n$ (cf. \cite[paragraph following (2.4)]{MT2012-2}). 
\end{remark}

% ----------------------------------------------------------------------

The main result of this article is that a compact containment condition holds for the sequence of approximating populations $(X^n)_{n \in \NN}$:

% ----------------------------------------------------------------------

\begin{theorem}
\label{THM:MT-extended}
For all $T, \epsilon>0$ there exists $\CK \subset \CM_F(\DD_E)$ relatively compact, such that
\eqn{ 
  \sup_{n \in \NN} \PP( \exists t \in [0,T], X_t^n \not\in \CK) \leq \epsilon. 
}
\end{theorem}

% ----------------------------------------------------------------------

\begin{proof}
By reasoning as in the \textit{Sketch of proof of \cite[Lemma 3.5]{MT2012-2}} it is enough to show the statement of Proposition~\ref{PRO:MT-extended} below, corresponding in spirit to item (i) of \cite[Lemma 3.5]{MT2012-2}. Note that due to the use of finite measures $X_t^n \in \CM_F(\DD_E)$ instead of probability measures, one needs to control the total mass, too, and this is clear from \eqref{unif-moment-bound}. \end{proof}
%
% ----------------------------------------------------------------------

\begin{proposition}
\label{PRO:MT-extended}
For all $T, \epsilon>0$ there exists $K \subset \DD_E$ compact such that if
\eqn{
  K_T = \{ y^t, y^{t-} | y \in K, t \in [0,T] \} \subset \DD_E,
}
then
\eqn{
\lbeq{statement-MT-i} 
  \sup_{n \in \NN} \PP( \exists t \in [0,T], X_t^n(K_T^c) > \epsilon) \leq \epsilon.
}
\end{proposition}

% ----------------------------------------------------------------------

The proof of Proposition~\ref{PRO:MT-extended} follows below.

% ----------------------------------------------------------------------

\begin{remark}[Assume without loss of generality $D \equiv 0$ and $U \equiv 0$ in \eqref{death-rate}]
\lbeq{rmk-assume-death-zero}
Recall the weakening of the assumption on the interaction kernel (see \eqref{ass-on-U} and the paragraph following it). When we decrease the death-rate we can introduce a coupling of the original historical process with a historical process with $D \equiv 0, U \equiv 0$ in such a way that the population of the former is a sub-population of the latter, uniformly over time. In what follows we will loosely call such a coupling ``dominating'' and a coupling where the coupled process yields sub-populations of the populations of the original process ``minorizing''. Once we prove \eqref{statement-MT-i} for the case $D \equiv 0$ and $U \equiv 0$ we therefore obtain \eqref{statement-MT-i} for $D, U$ satisfying \eqref{ass-on-U}. Note in particular that the scaling by $1/n$ in \eqref{def-X} is crucial for such a conclusion. 
\end{remark}

% ----------------------------------------------------------------------

One of the main steps to prove Proposition~\ref{PRO:MT-extended} is to establish the following result. The generalization to Polish spaces and more general mutation operators is the main challenge in comparison to \cite{MT2012-2}. The proof of Proposition~\ref{PRO:MT-extended} can be found in Section~\ref{SECTION:proof-main-prop}. The proof of Proposition~\ref{PRO:MT-Step-6} is postponed to Section~\ref{SECTION:proof-step-6}.

Denote by $K^T := \{ y^T | y \in K \} \subset \DD_E$ the set of the paths of $K$ stopped at time $T$.

% ----------------------------------------------------------------------

\begin{proposition}
\label{PRO:MT-Step-6}
For all $T, \epsilon>0$ there exists $K \subset \DD_E$ compact, such that
\eqn{
\lbeq{pro-MT-Step-6}
  \sup_{n \in \NN} \EE[ X_T^n((K^T)^c)] < \epsilon. 
}
\end{proposition}

% ----------------------------------------------------------------------

Note that it is enough to show that there exists $K \subset \DD_E$ relatively compact in Proposition~\ref{PRO:MT-Step-6}. The sets $K$ to be constructed in the proof of Proposition~\ref{PRO:MT-Step-6} are of a particular form, namely for $T>0$ we prove existence of $K \in \CD_T$ with $\CD_T$ as defined below. 

Before proceeding, the reader may want to have a look ahead at Definition~\ref{DEF:mod-of-cont} and Theorem~\ref{THM:rel-comp} where the notations $w'(y,\delta,T)$ respectively $w'(A,\delta,T)$ for the modulus of continuity of a path $y$ respectively a set $A$ and a criterion for relative compactness in $\DD_E$ are recalled from \cite{bEK2005}. 

% ----------------------------------------------------------------------

\begin{definition}
\label{def-special-rel-comp}
Let $\CD_T$ be the set of sets $K = K^T \subset \DD_E$ that satisfy: There exist $\Gamma_T \subset E$ compact and $(w'(\delta,T))_{\delta \in (0,1)} \in \RR_+ \cup \{\infty\}$ nondecreasing in $\delta$ with $\lim_{\delta \rightarrow 0} w'(\delta,T) = 0$ such that
\eqn{ 
  K = \{ y \in \DD_E: y=y^T, y(t) \in \Gamma_T \ \forall t \in [0,T], w'(y,\delta,T) \leq w'(\delta,T) \ \forall \delta \in (0,1) \}.
}
\end{definition}

% ----------------------------------------------------------------------

By the criterion for relative compactness in $\DD_E$ (cf. Theorem~\ref{THM:rel-comp}) all sets in $\CD_T$ are relatively compact in $\DD_E$. 

% ----------------------------------------------------------------------

We finish this section with a Lemma that yields a control on the modulus of continuity of the paths of the particles in the population. It is a direct consequence of Proposition~\ref{PRO:MT-extended}.

% ----------------------------------------------------------------------

\begin{lemma}
\label{LEM:app_to_mod_of_cont}
For all $T, \tau, \epsilon>0$ there exists $t_0=t_0(T,\tau,\epsilon)>0$ small enough such that
\eqn{
  \sup_{n \in \NN} \PP\big(\exists t \in [0,T], X_t^n \big( \big\{ y \in \DD_E: w'(y,t_0,t) \geq \tau \big\} \big) > \epsilon \big) \leq \epsilon. 
}
\end{lemma}

% ----------------------------------------------------------------------

\begin{proof}
By \eqref{statement-MT-i}, for all $T, \epsilon>0$ there exists $K \subset \DD_E$ compact such that  
\eqn{
  \sup_{n \in \NN} \PP( \exists t \in [0,T], X_t^n(K_T^c) > \epsilon) \leq \epsilon. 
}
As remarked in \cite[Sketch of the proof of Lemma 3.5]{MT2012-2}, $K_T$ is compact in $\DD_E$ (the reference \cite[Lemma 7.6]{bDP1991} holds for general Polish spaces $E$ as well). Recall Theorem~\ref{THM:rel-comp} to see that as a result of the compactness of $K_T$,
\eqn{
  \lim_{\delta \rightarrow 0} w'(K_T,\delta,T) = 0. 
}
Choose $t_0$ such that $w'(K_T,t_0,T) < \tau$ to conclude the claim. 
\end{proof}

% ----------------------------------------------------------------------

\begin{remark}[Application to evolving genealogies on marked metric measure spaces] 
In \cite{KW2014}, the compact containment result of Theorem~\ref{THM:MT-extended} as well as the control on the modulus of continuity as stated in Lemma~\ref{LEM:app_to_mod_of_cont} are applied in the context of evolving genealogies, modeled by means of marked metric measure spaces (mmm-spaces). Establishing relative compactness here requires, for example, a control on the number of balls of (genetic) radius $\epsilon$ necessary to cover the population. For an introduction to mmm-spaces the interested reader is referred to \cite{DGP2011}, for relative compactness see \cite[Proposition 7.1]{GPW2009} in the un-marked setup respectively \cite[Theorem 3 and Remark 2.5]{DGP2011} in the marked one. 

In \cite[Theorem 2]{GPW2013}, convergence of tree-valued Moran to Fleming-Viot dynamics is proven. Exponential rates are used to model the dynamics in the approximating population models. \cite{GPW2013} work in an ultra-metric setup where the genetic distance between two individuals alive at time $t$ equals twice the time to their most recent ancestor (cf. \cite[(2.20)]{GPW2013}). Hence, to obtain an $\epsilon$-coverage it remains to derive a bound on the number of most recent ancestors (mrca) at time $t-\epsilon$. In \cite{KW2014}, the metric under consideration is genetic distance instead: in the $n^{\mbox{th}}$-approximating population genetic distance is increased by $1/n$ at each birth with mutation. Hence, genetic distance of two individuals is counted in terms of births with mutation backwards in time to the mrca. In this non-ultrametric setup, the control over the whole path as provided by historical particle systems is particularly suitable. By interpreting genetic age of a particle as a trait, the control on the modulus of continuity of the historical path immediately translates into a control on genetic distance backwards in time.
\end{remark}

% ======================================================================

\section{Proof of Proposition~\ref{PRO:MT-extended}}
\label{SECTION:proof-main-prop}

% ======================================================================

\begin{proof}[Proof of Proposition~\ref{PRO:MT-extended}] 

For $T, \epsilon>0$ and $K \subset \DD_E$ compact let
\eqn{
  S_\epsilon^n := \inf\{ t \in \RR_+ | X_t^n(K_T^c) > \epsilon \} 
}
be the stopping time introduced in \cite[(3.18)]{MT2012-2} and rewrite 
\eqn{
  \PP( \exists t \in [0,T], X_t^n(K_T^c) > \epsilon) = \PP( S_\epsilon^n < T). 
}
Denote by $K^t := \{ y^t | y \in K \} \subset \DD_E$ the set of the paths of $K$ stopped at time $t$. To bound $\PP( S_\epsilon^n < T)$ by $\epsilon$, uniformly in $n \in \NN$, we have to control $X_t^n(K_T^c)$, that is the mass of the population outside of $K_T$, uniformly over the whole time-interval $t \in [0,T]$. The first step consists in introducing a more tractable quantity, namely instead of $K_T$ we follow \cite{MT2012-2} and focus on $K^T \subset K_T$ (note that if a path leaves $K_T$ it leaves $K^T$ as well) and decompose $\{ S_\epsilon^n < T \}$ into disjoint sets according to the behaviour of the population at the fixed final time $t=T$. We get
\eqn{
  \{ S_\epsilon^n < T \} \subset \big\{ X_T^n \big( (K^T)^c \big) > \frac{\epsilon}{2} \big\} \cup \big\{ S_\epsilon^n < T, X_T^n \big( (K^T)^c \big) \leq \frac{\epsilon}{2} \big\}. 
}

The probability of the first event can be bounded using Markov's inequality. The ensuing expectation $\EE[ X_T^n((K^T)^c) ]$ (at fixed time $T$) can be made arbitrarily small by choosing $K$ big enough as we will see later.

The bound on the second probability is the more involved. Reason as in \cite[Step 2]{MT2012-2} to see that to prove \eqref{statement-MT-i} it suffices to show that there exist $\eta \in (0,1), n_0 \in \NN$ both independent of $K \subset \DD_E$ such that for all $n \geq n_0$,
\eqn{
\lbeq{eta-claim}
  \PP\big( S_\epsilon^n < T, X_T^n((K^T)^c) \leq \frac{\epsilon}{2} \big) \leq \PP(S_\epsilon^n < T)(1-\eta) 
}
(cf. \cite[(3.21)]{MT2012-2} respectively Lemma~\ref{LEM:steps-3-5} below) and that one can choose $K \subset \DD_E$ compact big enough such that
\eqn{
\lbeq{exp-claim}
  \EE[ X_T^n((K^T)^c) ] < \frac{\epsilon^2 \eta}{2} 
}
(cf. \cite[(3.23)]{MT2012-2} respectively Proposition~\ref{PRO:MT-Step-6}) hold.

\textit{Outline of the remainder of the proof of Proposition~\ref{PRO:MT-extended}.} 
\textit{Steps 2--5} of the proof of \cite[Proposition 3.4]{MT2012-2} establish the claim we are interested in, that is the extension of the statement of \cite[Lemma 3.5(i)]{MT2012-2} (compare to Proposition~\ref{PRO:MT-extended}). We already recalled \textit{Step 2} above, leading up to inequalities \eqref{eta-claim}--\eqref{exp-claim} that remain to be shown. The claim of validity of the first inequality is formulated in Lemma~\ref{LEM:steps-3-5} below. The proof is an adaptation of \textit{Steps 3--5}. The change to the remaining \textit{Step 6}, that is the proof of \eqref{exp-claim} respectively Proposition~\ref{PRO:MT-Step-6} is the most involved due to allowing for a more general mutation kernel. The proof is therefore postponed to Section~\ref{SECTION:proof-step-6} below.

% ----------------------------------------------------------------------

\begin{remark}
\textit{Steps 3-5} of the proof of \cite[Proposition 3.4]{MT2012-2} contain a gap. The definition of $\eta$ is circular if one follows the reasoning in \cite[(3.26)--(3.42)]{MT2012-2} carefully.
In the alternative proof below we follow the ideas of \cite{MT2012-2} but avoid this recursive argument. 
As an additional result, the stronger assumption on the interaction kernel in \cite{MT2012-2}, namely $\underline{U}>0$ can be dropped as this is the only instance where it is used in \cite{MT2012-2}.
\end{remark}

% ----------------------------------------------------------------------

\begin{lemma}
\label{LEM:steps-3-5}
For $T, \epsilon>0$ and $K \subset \DD_E$ compact, there exist $\eta \in (0,1), n_0 \in \NN$ both independent of $K \subset \DD_E$ such that for all $n \geq n_0$,
\eqn{ 
  \PP\big( S_\epsilon^n < T, X_T^n((K^T)^c) \leq \frac{\epsilon}{2} \big) \leq \PP(S_\epsilon^T < T)(1-\eta). 
}
\end{lemma}

% ----------------------------------------------------------------------

\begin{proof}
Following the abstract reasoning of \textit{Step 3} of the proof of \cite[Proposition 3.4]{MT2012-2} up to and including equation \cite[(3.25)]{MT2012-2}, we conclude that it is enough to show that there exists $\eta \in (0,1), n_0 \in \NN$ large enough such that for $n \geq n_0$
\eqn{
\lbeq{to-show-eta}
  \PP\big( X_{S_\epsilon^n + (T-S_\epsilon^n)}^n(\{ y^{S_\epsilon^n} \not\in K^{S_\epsilon^n} \} ) \leq \frac{\epsilon}{2} \big| \CF_{S_\epsilon^n} \big) \leq 1 - \eta.
}

Now modify the reasoning in the remainder of \textit{Step 3} from \cite[(3.26)]{MT2012-2} onwards as follows: Couple the historical process $X^n$ to a minorizing process $(Z_t^n(dy))_{t \in \RR_+}$ with initial condition (cf. \cite[(3.27)]{MT2012-2})
\eqn{ 
  Z_{S_\epsilon^n}^n(dy) = \1_{y^{S_\epsilon^n} \not\in K^{S_\epsilon^n}} X^n_{S_\epsilon^n}(dy). 
}
Choose the birth rate $n r(t,y)$ as in \cite{MT2012-2} but change the death rate to $n r(t,y)+D_0$ with $D_0=D_0(T)>0$ a small enough constant to be chosen later on.
We now obtain instead of \cite[(3.28)]{MT2012-2} as an upper bound to the left hand side in \eqref{to-show-eta},
\eqn{
  1 - \PP\big( \inf_{s \in [S_\epsilon^n,T]} \langle Z_s^n,1 \rangle > \frac{\epsilon}{2} \big| \CF_{S_\epsilon^n} \big). 
}
%
%Instead of \cite[(3.29)]{MT2012-2}, 
It now remains to show that there exist $\eta \in (0,1), n_0 \in \NN$ such that 
\eqn{ 
\lbeq{ref-for-eta}
  \inf_{n \geq n_0} \PP\big( \inf_{s \in [S_\epsilon^n,T]} \langle Z_s^n,1 \rangle > \frac{\epsilon}{2} \big| \CF_{S_\epsilon^n} \big) \geq \eta. 
}
Follow the reasoning of \textit{Step 4} in \cite{MT2012-2}, the only difference being that we replace $\overline{D}+\overline{U} N$ by the constant $D_0$ and $2 \eta$ by $\eta$ throughout. Note in particular, that $\eta$ is finally defined as in \cite[(3.42)]{MT2012-2} but with the factor of $2$ replaced by $1$ on the right hand side. This leads directly up to \textit{Step 5}, where it remains to show that
\eqn{ 
  \PP_{z,r} \Big( \inf_{u \in [0,T]} \tilde{\CZ}_u \geq \frac{3\epsilon}{4} \Big) > 0 
}
for $(z,r) \in \DD_\RR \times [0,T]$ arbitrarily fixed and where $\tilde{\CZ}_\cdot$ is the diffusive limit of $\langle Z_{S_\epsilon^n+\cdot}^n , 1 \rangle$ as introduced above \cite[(3.34)]{MT2012-2} in \textit{Step 4}. Also note the characterization of $\tilde{\CZ}$ in \cite[(3.37)]{MT2012-2}. Following the reasoning of \cite[Step 5]{MT2012-2}, where we replace once more $\overline{D}+\overline{U} N$ by $D_0$, we obtain instead of the equation in between \cite[(3.44)--(3.45)]{MT2012-2},
\eqn{ 
  e^{\lambda \tilde{\CZ}_{t \wedge \zeta_M}} = e^{\lambda \epsilon} + \int_0^{t \wedge \zeta_M} \big( \tfrac{\lambda^2}{2} \rho(s) - \lambda D_0 \big) \tilde{\CZ}_s e^{\lambda \tilde{\CZ}_s} ds + \int_0^{t \wedge \zeta_M} \lambda \sqrt{ \rho(s) \tilde{\CZ}_s } e^{\lambda \tilde{\CZ}_s} dB_s 
}
for $\lambda>0$ with $\zeta_M = \inf\{ t \geq 0, \tilde{\CZ}_t \geq M \}$, $M>0$. Take expectations and choose $\lambda < D_0/\overline{R}$ (recall from \eqref{ass-r} that $0<\underline{R}<\overline{R}<\infty$ and from above \cite[(3.36)]{MT2012-2} that $2 \underline{R} \leq \rho(s) \leq 2 \overline{R}$)
to conclude analogously to \cite{MT2012-2} that $\EE( \mathrm{exp}(\lambda \tilde{Z}_{t \wedge \zeta_M}) ) \leq \mathrm{exp}(\lambda \epsilon)$. By choosing
\eqn{
\lbeq{ass-D-small}
  D_0 < \frac{4 \underline{R}}{T \overline{R}} 
}
we conclude as in \cite[(3.45)]{MT2012-2},
\eqn{ 
  \EE \Big( e^{\int_0^T \frac{D_0^2}{2 \rho(s)} \tilde{\CZ}_s ds} \Big) \leq \frac{1}{T} \int_0^T \EE \Big( e^{\frac{D_0^2 T}{4 \underline{R}} \tilde{\CZ}_s} \Big) ds \leq e^{\frac{D_0^2 T \epsilon}{4 \underline{R}}} < \infty. 
}
Now reason as in the remainder of \textit{Step 5} to obtain $\eta>0$. 
\end{proof}

% ----------------------------------------------------------------------

\textit{Conclusion of the proof of Proposition~\ref{PRO:MT-extended}.}
Taking Lemma~\ref{LEM:steps-3-5} and Proposition~\ref{PRO:MT-Step-6} together yields the claim. 
\end{proof}

% ======================================================================

\section{Proof of Proposition~\ref{PRO:MT-Step-6}}
\label{SECTION:proof-step-6}

% ======================================================================

\begin{proof}[Proof of  Proposition~\ref{PRO:MT-Step-6}] 

Coupling with a dominating historical particle system allows us to assume $b^n(t,y) = n r(t,y) + \overline{B}$ as birth rate (cf. \eqref{ass-r}) and $d^n(t,y) := n r(t,y)$ as death rate (cf. \eqref{death-rate}) at time $t$. 
Next construct the tree underlying $X^n$ analogously to \cite[Step 6]{MT2012-2} by pruning a Yule tree with traits in $E$. 

A particle of lineage $y$ at time $t$ gives two offspring (one is the parent, one the child) at rate $b^n(t,y)+d^n(t,y)$. One has lineage $y$ and the other has lineage $(y | t | h)$ (recall \eqref{concatenated-path}), where $h$ is distributed following 
\eqn{
\lbeq{jump-sizes-child}
  K^n(x,dh) := p \alpha_n(x,dh) + (1-p) \delta_x(dh)
}
with $x=y_{t-}$ (compare \cite[(2.5)]{MT2012-2}). Using Harris-Ulam-Neveu's notation to label the particles (see e.g. \cite{bD1993}), we denote by $Y^{n,\alpha}$ for $\alpha \in \CI = \cup_{m=0}^{+\infty} \{0,1\}^{m+1}$ the lineage of the particle with label $\alpha$.

% ----------------------------------------------------------------------

\begin{remark}[Clarification of notation] 
\label{RMK:lineage-nota-probl}
The lineage of the particle with label $\alpha$ does only record the lineage of the particle until the random time $S_{|\alpha|+1}$ (cf. \eqref{lineage}). To regard particles as individuals alive indefinitely, identify the lineage of the particle with label $\alpha$ with the lineage of the particle $(\alpha,\beta)$ with $\beta=(0,\ldots,0)$ and $|\beta| \rightarrow \infty$.
\end{remark}

% ----------------------------------------------------------------------

Particles descending from the same individual at time $0$ are exchangeable and the common distribution of the process $Y^{n,\alpha}$ (in the new notation) is the one of a pure jump process on $E$, where the jumps occur at rate $b^n(t,y)+d^n(t,y) = 2nr(t,y) + \overline{B}$ and where the new traits are distributed according to the probability measure 
\eqn{
\lbeq{jump-sizes-kernel}
  \tfrac{1}{2} \delta_{y_{t-}}(dh) + \tfrac{1}{2} K^n(y_{t-},dh)
}
(with probability $1/2$ we pick the parent with probability $1/2$ the child at the time of birth of an offspring). We denote by $\PP_x^n$ its distribution starting from $x \in E$.

At each node of the Yule tree, an independent pruning is made: the offspring are kept with probability $p(n) := b^n(t,y) / (b^n(t,y)+d^n(t,y))$ and are erased otherwise.
%
% ----------------------------------------------------------------------

Following \cite{MT2012-2}, let us denote by $V_t^n$ the set of individuals alive at time $t$ and write $\alpha \succ i$ to say that the individual $\alpha$ is a descendant of the individual $i$. Recall that $N_0^n$ is the number of individuals present at time $0$. Let
\eqn{
  \Sigma_i :=  \sum_{\alpha \succ i} \EE \big[ \PP \big( \alpha \in V_T^n \big| Y^{n,\alpha} \big) \1_{\{ (Y^{n,\alpha})^T \not\in K^T \}} \big]
}
so that for $K \subset \DD_E$ relatively compact,
\eqn{
\lbeq{control-time-dist-MT}
  \EE[ X_T^n((K^T)^c)] 
  = \EE \Big[ \frac{1}{n} \sum_{i=1}^{N_0^n} \Sigma_i \Big].
}
%
% ----------------------------------------------------------------------
 
\begin{remark}
In \textit{Step 6} of the proof of \cite[Proposition 3.4]{MT2012-2}, an error occurs when rewriting the expectation corresponding to \eqref{control-time-dist-MT} above. The pruning of the Yule tree is not independent of the process $Y^{n,\alpha}$ in so far as the pruning parameter depends on the path of the particle. In what follows a new proof is given that further allows to handle Polish trait spaces and more general mutation kernels.
\end{remark}

% ----------------------------------------------------------------------

Next, recall from \cite{bEK2005} a criterion for relative compactness in $\DD_E$ and the definition of modulus of continuity used therein. 
%
% ----------------------------------------------------------------------

\begin{definition}[modulus of continuity, {\cite[III.6.(6.2)]{bEK2005}}]
\label{DEF:mod-of-cont}
Let $(E,r)$ denote a metric space. For $x \in \DD_E, \delta>0$ and $T>0$, define
\eqn{
\lbeq{def-mod-cont}
  w'(x,\delta,T) = \inf_{\{t_i\}} \max_i \sup_{s,t \in [t_{i-1},t_i)} r(x(s),x(t)),
}
where $\{t_i\}$ ranges over all partitions of the form $0=t_0<t_1< \cdots <t_{n-1}<T \leq t_n$ with $\min_{1 \leq i \leq n} (t_i-t_{i-1}) > \delta$ and $n \geq 1$. Note that $w'(x,\delta,T)$ is nondecreasing in $\delta$ and in $T$.
\end{definition}

% ----------------------------------------------------------------------

\begin{theorem}[criterion for relative compactness in $\DD_E$, {\cite[III.6.Theorem 6.3 and Remark 6.4]{bEK2005}}]
\label{THM:rel-comp}
Let $(E,r)$ be complete. Then $A \subset \DD_E$ is relatively compact if and only if the following two conditions hold:
\begin{enumerate}
\item[(a)]
For each $T>0$ there exist a compact set $\Gamma_T \subset E$ such that $x(t) \in \Gamma_T$ for $0 \leq t \leq T$ and all $x \in A$.
\item[(b)]
For each $T>0$,
\eqn{
\lbeq{rc-cond-mod-cont}
  \lim_{\delta \rightarrow 0} w'(A,\delta,T) := \lim_{\delta \rightarrow 0} \sup_{x \in A} w'(x,\delta,T) = 0.
}
\end{enumerate}
\end{theorem}

% ----------------------------------------------------------------------

Recall Definition~\ref{def-special-rel-comp} and the comment following it. In what follows it is therefore sufficient to prove that for $T>0$ fixed there exist a compact set $\Gamma_T \subset E$ and $(w'(\delta,T))_{\delta \in (0,1)} \in \RR_+ \cup \{\infty\}$ nondecreasing in $\delta$ with $\lim_{\delta \rightarrow 0} w'(\delta,T) = 0$ such that 
\eqn{
\lbeq{def-compacts-tightness}
  K := \{ y \in \DD_E: y=y^T, y(t) \in \Gamma_T \ \forall t \in [0,T], w'(y,\delta,T) \leq w'(\delta,T) \ \forall \delta \in (0,1) \}
}
satisfies \eqref{pro-MT-Step-6}.

% ----------------------------------------------------------------------

\textit{Continuation of the proof of Proposition~\ref{PRO:MT-Step-6}.}
Let $N_T^{n,\alpha}$ denote the number of jumps of the particle with label $\alpha$ up to time $T$ (recall Remark~\ref{RMK:lineage-nota-probl}). Then 
\eqn{
  \PP \big( \alpha \in V_T^n \big| Y^{n,\alpha} \big)
  \leq \1_{\{ N_T^{n,\alpha} = |\alpha| \}} \Big( \frac{n \underline{R} + \overline{B}}{2n \underline{R} + \overline{B}} \Big)^{|\alpha|}.
}
Therefore, summing over $|\alpha|$, we get for some $c>0$,
\eqn{
  \Sigma_i 
  \leq \sum_{k=0}^\infty \Big( 1+\frac{c}{n} \Big)^k 2^{-k} \sum_{\alpha \succ i, |\alpha|=k} \PP\big( N_T^{n,\alpha} = k , (Y^{n,\alpha})^T \not\in K^T \big).
}

Let $Y^n$ be a process that starts in $X_0^i \in E$, the initial position of individual $i \in \{1,\ldots,N_0^n\}$ and is distributed according to $\PP^n_{X_0^i}$. Denote by $N_T(Y^n)$ the number of jumps of $Y^n$ up to time $T$. Then, for any $A>0$,
\eqan{
\lbeq{bound-on-Sigma-i}  
  \Sigma_i
  &\leq \sum_{k=0}^\infty e^{ck/n} \PP\big( N_T(Y^n) = k , (Y^n)^T \not\in K^T \big) \\
  &\leq e^{cA} \PP\big( (Y^n)^T \not\in K^T \big) + \EE\big[ e^{cN_T(Y^n)/n} \1_{\{ N_T(Y^n) > An \}} \big]. \nn
}
Let $\overline{Y}^n$ be a coupled jump-process which has the same sequence of jumps as $Y^n$ but jumps at dominating rate $2n\overline{R}+\overline{B}$. Then the coupling can be constructed such that the inter-jump-times of $\overline{Y}^n$ minorize those of $Y^n$. The fact that these times are equal or smaller implies that by definition of $K$, $\PP\big( (Y^n)^T \not\in K^T \big) \leq \PP\big( (\overline{Y}^n)^T \not\in K^T \big)$ and $N_T(Y^n) \leq N_T(\overline{Y}^n)$, the latter being $\mbox{Pois}(\lambda_n)$ with $\lambda_n := T(2n\overline{R}+\overline{B})$. Then there exist constants $C_1, C_2>0$ such that for any $\epsilon_0>0$ we may now choose $A$ large enough so that
\eqan{
  & \EE\big[ e^{cN_T(Y^n)/n} \1_{\{ N_T(Y^n) > An \}} \big]
  \leq \EE\big[ e^{cN_T(\overline{Y}^n)/n} \1_{\{ N_T(\overline{Y}^n) > An \}} \big] \\
  &\leq e^{\lambda_n (e^{c/n}-1)} \PP\big( \mbox{Pois}( \lambda_n e^{c/n} ) \geq An \big)
  \leq C_1 \PP\big( \mbox{Pois}(C_2n) \geq An \big)
  < \epsilon_0. \nn
}
Put this back into \eqref{bound-on-Sigma-i} and \eqref{control-time-dist-MT} to obtain
\eqn{
\lbeq{final-bound}
  \EE[ X_T^n((K^T)^c)] 
  \leq e^{cA} \EE \Big[ \int_{\DD_E} X_0^n(dy) \overline{\PP}^n_{y_0}\big( \{ y: y^T \not\in K^T \} \big) \Big] + \epsilon_0 \EE \Big[ \langle X_0^n , 1 \rangle \Big],
}
where $\overline{\PP}^n_{y_0}$ denotes the distribution of $\overline{Y}^n$ starting in $y_0$. 

Choose $A$ big enough such that the second term in \eqref{final-bound} is $\epsilon/2$ at most, uniformly in $n \in \NN$. Keep $A$ fixed and use \eqref{tightness-initial} and Hypothesis~\ref{HYP:mut-kernel} to get the required bound in Proposition~\ref{PRO:MT-Step-6}. Here we note that the process $Y^n$ of Hypothesis~\ref{HYP:mut-kernel} jumps according to the kernel $\alpha_n(x,dh)$ at rate $n$, whereas the process $\overline{Y}^n$ jumps under $\overline{\PP}^n_{y_0}$ at rate $2n\overline{R}+\overline{B}$ according to the jump kernel in \eqref{jump-sizes-kernel}. The change in the rate amounts to a time change only. Replacing jumps by jumps of size zero increases the chances to stay inside the relatively compact set $K$ (cf. Theorem~\ref{THM:rel-comp}). 
\end{proof}

% ======================================================================

%%                                                               %%
%% Use the two commands below for producing your bibliography    %%
%% with bibtex, then comment again the commands and include the  %%
%% content of the .bbl file in this file below the commands.     %%
%%                                                               %%

%\bibliographystyle{amsplain}
%\bibliography{bib}

% add below the content of your .bbl file produced by bibtex.

\providecommand{\bysame}{\leavevmode\hbox to3em{\hrulefill}\thinspace}
\providecommand{\MR}{\relax\ifhmode\unskip\space\fi MR }
% \MRhref is called by the amsart/book/proc definition of \MR.
\providecommand{\MRhref}[2]{%
  \href{http://www.ams.org/mathscinet-getitem?mr=#1}{#2}
}
\providecommand{\href}[2]{#2}

%%                                                               %%
%% You may add acknowledgments (optional).                       %%
%%                                                               %%

\ACKNO{Many thanks go to Wolfgang L\"{o}hr for helpful discussions. Further thanks go to Viet Chi Tran for feedback on the underlying article. Finally, the author wishes to thank a referee for a number of suggestions that helped to improve the exposition of this article and streamline proofs. This research was supported by the DFG through the SPP Priority Programme 1590.}

%%                                                               %%
%% You have reached the end of your document.                    %%
%%                                                               %%

\end{document}